\documentclass[10pt,a4paper,oneside]{article}
\usepackage[english]{babel}
\usepackage{a4wide,amsmath,amsthm,amssymb,url,graphicx,xspace,algorithm,algorithmic,pgf}
\usepackage[latin1]{inputenc}
\usepackage{enumitem}
\usepackage{multirow}
\usepackage{hyperref}
\usepackage{tabularx}
\usepackage{tikz}
\usepackage{allrunes}

\DeclareMathOperator{\PG}{PG}
\DeclareMathOperator{\PGL}{PGL}

\DeclareMathOperator{\row}{row}
\DeclareMathOperator{\rk}{rk}

\theoremstyle{definition}
\newtheorem{theorem}{Theorem}[section]
\newtheorem{lemma}[theorem]{Lemma}
\newtheorem{definition}[theorem]{Definition}
\newtheorem{remark}[theorem]{Remark}

\newtheorem{example}[theorem]{Example}

\newcommand{\comments}[1]{}

\author{Maarten De Boeck\footnote{Address: Ghent University, Department of Mathematics, Krijgslaan 281, 9000 Gent, Flanders, Belgium. \newline Email address: \{mdeboeck,ls\}@cage.ugent.be, Website: http://cage.ugent.be/$\sim$ls}, Leo Storme\footnotemark[1] and Andrea \v{S}vob\footnote{Address: University of Rijeka, Department of Mathematics, Radmile Matej\v{c}i\'c 2, 51000 Rijeka, Croatia.  \newline Email address: asvob@math.uniri.hr, Website: http://www.math.uniri.hr/$\sim$asvob/}}
\title{The Cameron-Liebler problem for sets}
\date{}
\begin{document}
\maketitle

\begin{abstract}
 Cameron-Liebler line classes and Cameron-Liebler $k$-classes in $\PG(2k+1,q)$ are currently receiving a lot of attention. Here, links with the Erd\H{o}s-Ko-Rado results in finite projective spaces occurred. We introduce  here in this article the similar problem on Cameron-Liebler classes of sets, and solve this problem completely, by making links to the classical Erd\H{o}s-Ko-Rado result on sets. We also present a characterisation theorem for the Cameron-Liebler classes of sets.
\end{abstract}


\section{Introduction}\label{sec:introduction}

In \cite{cl}, Cameron and Liebler investigated the orbits of the projective groups $\PGL(n+1,q)$. For this purpose they introduced line classes in the projective space $\PG(3,q)$ with a specific property, which afterwards were called \emph{Cameron-Liebler line classes}. A Cameron-Liebler line class $\mathcal{L}$ with parameter $x$ in $\PG(3,q)$ is a set of $x(q^{2}+q+1)$ lines in $\PG(3,q)$ such that any line $\ell\in\mathcal{L}$ meets precisely $x(q+1)+q^{2}-1$ lines of $\mathcal{L}$ in a point and such that any line $\ell\notin\mathcal{L}$ meets precisely $x(q+1)$ lines of $\mathcal{L}$ in a point.
\par Many equivalent characterisations are known, of which we present one. For an overview we refer to \cite[Theorem 3.2]{drudge}. A \emph{line spread} of $\PG(3,q)$ is a set of lines that form a partition of the point set of $\PG(3,q)$, i.e. each point of $\PG(3,q)$ is contained in precisely one line of the line spread. The lines of a line spread are necessarily pairwise skew. Now a line set $\mathcal{L}$ in $\PG(3,q)$ is a Cameron-Liebler line class with parameter $x$ if and only if it has $x$ lines in common with every line spread of $\PG(3,q)$.
\par The central problem for Cameron-Liebler line classes in $\PG(3,q)$, is to determine for which parameters $x$ a Cameron-Liebler line class exists, and to classify the examples admitting a given parameter $x$. 
Constructions of Cameron-Liebler line classes and characterisation results were obtained in \cite{bd,cl,ddmr,TF:14,met,rod}. Recently several results were obtained through a new counting technique, see \cite{gmet,gm,met2}. A complete classification is however not in sight.
\par Also recently, Cameron-Liebler $k$-classes in $\PG(2k+1,q)$ were introduced in \cite{rsv} and Cameron-Liebler line classes in $\PG(n,q)$ were introduced in \cite{gm}. Both generalise the classical Cameron-Liebler line classes in $\PG(3,q)$.
\par Before describing the central topic of this article, we recall the concept of a $q$-analogue. In general a $q$-analogue is a mathematical identity, problem, theorem,...,  that depends on a variable $q$ and that generalises a known identity, problem, theorem,..., to which it reduces in the (right) limit $q\to1$. In a combinatorial/geometrical setting it often arises by replacing a set and its subsets by a vector space and its subspaces. E.g. the $q$-binomial theorem is a $q$-analogue of the classical binomial theorem. In recent years there has been a lot of attention for $q$-analogues, see \cite{bbsw} amongst others.
\par The Cameron-Liebler problem has not arisen as a $q$-analogue of a problem on sets, but has an interesting counterpart on sets that we will describe and investigate in this article. The definition builds on the spread definition of the classical Cameron-Liebler line classes and uses a classical set counterpart for spreads in a projective space. A subset of size $k$ of a set will be called a {\em $k$-subset} or shortly a {\em $k$-set}.

\begin{definition}
  A \emph{$k$-uniform partition} of a finite set $\Omega$, with $|\Omega|=n$ and $k\mid n$, is a set of pairwise disjoint $k$-subsets of $\Omega$ such that any element of $\Omega$ is contained in precisely one of the $k$-subsets.
\end{definition}

Necessarily, a $k$-uniform partition of a finite set $\Omega$, with $|\Omega|=n$, contains $\frac{n}{k}$ different $k$-subsets. This definition now allows us to present the definition of a Cameron-Liebler class of $k$-sets.

\begin{definition}\label{clclass}
  Let $\Omega$ be a finite set with $|\Omega|=n$ and let $k$ be a divisor of $n$. A \emph{Cameron-Liebler class of $k$-sets with parameter $x$} is a set of $k$-subsets of $\Omega$ which has $x$ different $k$-subsets in common with every $k$-uniform partition of $\Omega$.
\end{definition}

Note that the $q$-analogue of the above definition is actually a Cameron-Liebler $(k-1)$-class in $\PG(n-1,q)$, a concept that has not been discussed before, but which is a straightforward generalisation of the Cameron-Liebler classes that have already been discussed.
\par We present two results on these Cameron-Liebler classes of subsets. In Theorem \ref{char} we show that also for Cameron-Liebler classes of subsets many equivalent characterisations can be found. The second main theorem of this paper is the following classification result.
\begin{theorem}\label{main}
  Let $\Omega$ be a finite set with $|\Omega|=n$ and let $\mathcal{L}$ be a Cameron-Liebler class of $k$-sets with parameter $x$ in $\Omega$, $k\geq2$. If $n\geq 3k$ and $\mathcal{L}$ is nontrivial, then either $x=1$ and $\mathcal{L}$ is the set of all $k$-subsets containing a fixed element or $x=\frac{n}{k}-1$ and $\mathcal{L}$ is the set of all $k$-subsets not containing a fixed element.
\end{theorem}

\section{The classification result}\label{sec:classification}

The next result is the Erd\H{o}s-Ko-Rado theorem, a classical result in combinatorics.

\begin{theorem}[{\cite[Theorem 1]{ekr} and \cite{w}}]\label{ekr}
  If $\mathcal{S}$ is a family of $k$-subsets in a set $\Omega$ with $|\Omega|=n$ and $n\geq2k$, such that the elements of $\mathcal{S}$ are pairwise not disjoint, then $|\mathcal{S}|\leq\binom{n-1}{k-1}$. Moreover, if $n\geq2k+1$, then equality holds if and only if $\mathcal{S}$ is the set of all $k$-subsets through a fixed element of $\Omega$.
\end{theorem}

\begin{lemma}\label{basic}
  Let $\Omega$ be a finite set with $|\Omega|=n$, and let $\mathcal{L}$ be a Cameron-Liebler class of $k$-sets with parameter $x$ in $\Omega$, with $k\mid n$.
  \begin{enumerate}
    \item The number of $k$-uniform partitions of $\Omega$ equals $\displaystyle\frac{n!}{\left(\frac{n}{k}\right)!(k!)^{\frac{n}{k}}}$.
    \item The number of $k$-sets in $\mathcal{L}$ equals $x\binom{n-1}{k-1}$.
    \item The set $\overline{\mathcal{L}}$ of $k$-subsets of $\Omega$ not belonging to $\mathcal{L}$ is a Cameron-Liebler class of $k$-sets with parameter $\frac{n}{k}-x$.
  \end{enumerate}
\end{lemma}
\begin{proof}
  \begin{enumerate}
    \item With every permutation (ordering) $\sigma$ of the $n$ elements of $\Omega$, we can construct a partition $P_{\sigma}$ in the following way: for every $i=1,\dots,\frac{n}{k}$ the elements on the positions $(i-1)k+1,(i-1)k+2\dots,ik$ form a $k$-subset of $\Omega$, and these $\frac{n}{k}$ subsets are pairwise disjoint and form thus a $k$-uniform partition. Now, every partition can arise from $\left(\frac{n}{k}\right)!(k!)^{\frac{n}{k}}$ different permutations as the $\frac{n}{k}$ subsets can be permuted and each of these $k$-subsets can be permuted internally.
    \item We perform a double counting of the tuples $(C,P)$, with $C\in\mathcal{L}$, $P$ a $k$-uniform partition and $C$ a $k$-set in $P$. We find that
    \[
      |\mathcal{L}|\frac{(n-k)!}{\left(\frac{n}{k}-1\right)!(k!)^{\frac{n-k}{k}}}=x\frac{n!}{\left(\frac{n}{k}\right)!(k!)^{\frac{n}{k}}}\quad\Rightarrow\quad|\mathcal{L}|=x\frac{n!\:k}{(n-k)!\:k!\:n}=x\binom{n-1}{k-1}\;.
    \]
    \item Since every $k$-uniform partition of $\Omega$ contains $x$ subsets belonging to $\mathcal{L}$, it contains $\frac{n}{k}-x$ subsets belonging to $\overline{\mathcal{L}}$.\qedhere
  \end{enumerate}
\end{proof}

\begin{example}\label{vb}
  Let $\Omega$ be a finite set with $|\Omega|=n$, and assume $k\mid n$. We give some examples of Cameron-Liebler classes of $k$-sets with parameter $x$. Note that $0\leq x\leq\frac{n}{k}$.
  \begin{itemize}
    \item The empty set is obviously a Cameron-Liebler class of $k$-sets with parameter $0$, and directly or via the last property in Lemma \ref{basic} it can be seen that the set of all $k$-subsets of $\Omega$ is a Cameron-Liebler class of $k$-sets with parameter $\frac{n}{k}$. These two examples are called the \emph{trivial} Cameron-Liebler classes of $k$-sets.
    \item Let $p$ be a given element of $\Omega$. The set of $k$-subsets of $\Omega$ containing $p$ is a Cameron-Liebler class of $k$-sets with parameter $1$. Indeed, in every $k$-uniform partition of $\Omega$ there is exactly one $k$-subset containing $p$.
    \par Again using the last property of Lemma \ref{basic}, we find that the set of all $k$-subsets of $\Omega$ not containing the element $p$ is a Cameron-Liebler class of $k$-sets with parameter $\frac{n}{k}-1$.
  \end{itemize}
\end{example}

In the introduction we already mentioned that many equivalent characterisations for Cameron-Liebler classes in $\PG(3,q)$ are known. In Theorem \ref{char} we show that this is also true for Cameron-Liebler classes of subsets. We did not mention the equivalent characterisations for the Cameron-Liebler sets in $\PG(3,q)$, but they arise as the $q$-analogues of the characterisations in Theorem \ref{char}.
\par Before stating this theorem, we need to introduce some concepts. The \emph{incidence vector} of a subset $A$ of a set $S$ is the vector whose positions correspond to the elements of $S$, with a one on the positions corresponding to an element in $S$ and a zero on the other positions. Below we will use the incidence vector of a family of $k$-subsets of a set $\Omega$: as this family is a subset of the set of all $k$-subsets of $\Omega$, each position corresponds to a $k$-subset of $\Omega$. For any vector $v$ whose positions correspond to elements in a set, we denote its value on the position corresponding to an element $a$ by $(v)_{a}$. The all-one vector will be denoted by $j$.
\par Given a set $\Omega$, we also need the \emph{incidence matrix of elements and $k$-subsets}. This is the $|\Omega|\times\binom{|\Omega|}{k}$-matrix whose rows are labelled with the elements of $\Omega$, whose columns are labelled with the $k$-sets of $\Omega$ and whose entries equal $1$ if the element corresponding to the row is contained in the $k$-set corresponding to the column, and zero otherwise. The \emph{Kneser matrix} or \emph{disjointness matrix} of $k$-sets in $\Omega$ is the $\binom{|\Omega|}{k}\times\binom{|\Omega|}{k}$-matrix whose rows and columns are labelled with the $k$-sets of $\Omega$ and whose entries equal $1$ if the $k$-set corresponding to the row and the $k$-set corresponding to the column are disjoint, and zero otherwise\footnote{This Kneser graph is often introduced as the adjacency matrix of the disjointness graph, but as there is no need here to introduce this graph, we introduced it directly.}.

We will need the following result about the Kneser matrix.

\begin{lemma}\label{eigkneser}
  Let $\Omega$ be a finite set with $|\Omega|=n$ and let $K$ be the Kneser matrix of the $k$-sets in $\Omega$. The eigenvalues of $K$ are given by $\lambda_{j}=(-1)^{j}\binom{n-k-j}{k-j}$, $j=0,\dots,k$, and the multiplicity of the eigenvalue $\lambda_{j}$ is $\binom{n}{j}-\binom{n}{j-1}$.
\end{lemma}

A direct but lengthy proof of this result can be found in \cite{rei}. The original proofs of this result can be found in \cite[Theorem 4.6]{del} and \cite[Section 2(d)]{yfh}, but both use the theory of association schemes, which we did not introduce in this article. The Kneser matrix is a matrix of the Johnson scheme.
\par Now we can present a theorem with many equivalent characterisations of Cameron-Liebler classes of $k$-subsets.

\begin{theorem}\label{char}
  Let $\Omega$ be a finite set with $|\Omega|=n$, and let $k$ be a divisor of $n$. Let $\mathcal{L}$ be a set of $k$-subsets of $\Omega$ with incidence vector $\chi$. Denote $\frac{|\mathcal{L}|}{\binom{n-1}{k-1}}$ by $x$. Let $C$ be the incidence matrix of elements and $k$-subsets in $\Omega$ and let $K$ be the Kneser matrix of $k$-sets in $\Omega$. The following statements are equivalent. 
  \begin{itemize}
    \item[(i)] $\mathcal{L}$ is a Cameron-Liebler class of $k$-sets with parameter $x$.
    \item[(ii)] $\mathcal{L}$ has $x$ different $k$-subsets in common with every $k$-uniform partition of $\Omega$.
    \item[(iii)] For each fixed $k$-subset $\pi$ of $\Omega$, the number of elements of $\mathcal{L}$ disjoint from $\pi$ equals $(x-(\chi)_{\pi})\binom{n-k-1}{k-1}$.
    \item[(iv)] The vector $\chi-\frac{kx}{n}j$ is contained in the eigenspace of $K$ for the eigenvalue $-\binom{n-k-1}{k-1}$.
    \item[(v)] $\chi\in\row(C)$.
    \item[(vi)] $\chi\in(\ker(C))^{\perp}$.
  \end{itemize}
\end{theorem}
\begin{proof}
  If $k=n$, then there is only one $k$-set in $\Omega$. This one $k$-set could be contained in $\mathcal{L}$ (case $x=1$) or not (case $x=0$), but in both cases all the statements are valid. Note that $K$ is the $1\times 1$ zero matrix. From now on we assume that $n\geq2k$.
  \par By Definition \ref{clclass} statements (i) and (ii) are equivalent. Now, we assume that statement (ii) is valid. Let $\pi$ be a fixed $k$-subset of $\Omega$ and denote the number of elements of $\mathcal{L}$ disjoint from $\pi$ by $N$. Double counting the set of tuples $(\sigma,U)$, with $\sigma$ a $k$-set of $\Omega\setminus\pi$ and $U$ a $k$-uniform partition of $\Omega$ containing $\pi$ and $\sigma$, we find
  \begin{align*}
    &N\frac{(n-2k)!}{\left(\frac{n-2k}{k}\right)!(k!)^{\frac{n-2k}{k}}}=\frac{(n-k)!}{\left(\frac{n-k}{k}\right)!(k!)^{\frac{n-k}{k}}}(x-(\chi)_{\pi})\\
    \Leftrightarrow\quad &N=(x-(\chi)_{\pi})\frac{(n-k)!}{(n-2k)!k!}\frac{k}{n-k}=(x-(\chi)_{\pi})\binom{n-k-1}{k-1}\;,
  \end{align*}
  since $U$ contains $x$ elements of $\mathcal{L}$ disjoint from $\pi$ if $\pi\notin\mathcal{L}$, and $x-1$ elements of $\mathcal{L}$ disjoint from $\pi$ if $\pi\in\mathcal{L}$.
  \par If statement (iii) is valid for the set $\mathcal{L}$, then it follows immediately that $K\chi=\binom{n-k-1}{k-1}(xj-\chi)$ since on the position corresponding to the $k$-set $\pi$ the vector $K\chi$ has the number of elements in $\mathcal{L}$ disjoint from $\pi$; the vector $\binom{n-k-1}{k-1}(xj-\chi)$ has $\binom{n-k-1}{k-1}(x-(\chi)_{\pi})$ on this position. We also know that $Kj=\binom{n-k}{k}j$ since for every $k$-set in $\Omega$ there are $\binom{n-k}{k}$ different $k$-sets disjoint to it. We find that
  \begin{align*}
    K\left(\chi-\frac{kx}{n}j\right)&=\binom{n-k-1}{k-1}(xj-\chi)-\frac{kx}{n}\binom{n-k}{k}j\\
    &=-\binom{n-k-1}{k-1}\left(\chi-xj+\frac{n-k}{k}\frac{kx}{n}j\right)=-\binom{n-k-1}{k-1}\left(\chi-\frac{kx}{n}j\right)\;.
  \end{align*}
  Hence, $\chi-\frac{kx}{n}j$ is a vector in the eigenspace of $K$ for the eigenvalue $-\binom{n-k-1}{k-1}$.
  \par We assume that statement (iv) is valid. Let $V$ be the eigenspace of $K$ related to the eigenvalue $-\binom{n-k-1}{k-1}$. Then $\chi\in\left\langle j\right\rangle\oplus V$. By Example \ref{vb} the set of all $k$-sets containing a fixed element $p$ is a Cameron-Liebler set with parameter $1$. Denote the incidence vector of this Cameron-Liebler set by $v_{p}$. Then $v^{t}_{p}$ is the row of $C$ corresponding to $p$. Since (i) implies (iv), through (ii) and (iii), we know that $v^{t}_{p}-\frac{k}{n}j$ is a vector in the eigenspace of $K$ for the eigenvalue $-\binom{n-k-1}{k-1}$. So $v^{t}_{p}\in\left\langle j\right\rangle\oplus V$. It follows that $\row(C)$ is a subspace of $\left\langle j\right\rangle\oplus V$.
  \par The set of all $k$-sets of $\Omega$ forms a $k-(n,k,1)$ block design, hence also a $2-(n,k,\binom{n-2}{k-2})$ block design, whose incidence matrix is $C$. From Fisher's inequality (see \cite{bose,fis}) it follows that $\rk(C)=n$. Using Lemma \ref{eigkneser} we know that $\dim(\left\langle j\right\rangle\oplus V)=n$. Since $\row(C)\subseteq\left\langle j\right\rangle\oplus V$ and $\dim(\row(C))=n=\dim(\left\langle j\right\rangle\oplus V)$, the subspaces $\row(C)$ and $\left\langle j\right\rangle\oplus V$ are equal. Consequently, $\chi\in\row(C)$.
  \par Statements (v) and (vi) are clearly equivalent as $\row(C)=(\ker(C))^{\perp}$.
  \par Finally, we assume that statement (vi) is valid. Let $U$ be a $k$-uniform partition of $\Omega$ and let $\chi_{U}$ be its incidence vector. It is clear that $C\chi_{U}=j$ since each element of $\Omega$ is contained in precisely one element of $U$. Since $Cj=\binom{n-1}{k-1}j$ it follows that $\chi_{U}-\frac{1}{\binom{n-1}{k-1}}j\in\ker(C)$. From the assumption that (vi) is valid, it follows that $\chi$ and $\chi_{U}-\frac{1}{\binom{n-1}{k-1}}j$ are orthogonal. Hence,
  \[
    |\mathcal{L}\cap U|=\left\langle\chi,\chi_{U}\right\rangle=\frac{1}{\binom{n-1}{k-1}}\left\langle\chi,j\right\rangle=\frac{|\mathcal{L}|}{\binom{n-1}{k-1}}=x\;,
  \]
  which proves (ii).
\end{proof}

The main theorem of this paper, Theorem \ref{main}, states that the examples in Example \ref{vb} are the only examples of Cameron-Liebler classes of $k$-sets, in case $n\geq3k$. We note that only four parameter values are admissable. The next lemmata show this result.

\begin{lemma}\label{main1}
  Let $\mathcal{L}$ be a nontrivial Cameron-Liebler class  of $k$-sets with parameter $x$ in a set $\Omega$ of size $n\geq3k$, $x<\frac{n}{k}-1$ and $k\geq2$. Then, $\mathcal{L}$ is the set of all $k$-sets through a fixed element and $x=1$.
\end{lemma}
\begin{proof}
  It follows immediately from the definition of a Cameron-Liebler class of $k$-sets with parameter $x$ that there are $x$ pairwise disjoint $k$-sets in $\mathcal{L}$. Let $H_{1},\dots,H_{x}\in\mathcal{L}$ be $x$ pairwise disjoint $k$-sets, and let $\mathcal{S}_{i}$ be the set of $k$-sets in $\mathcal{L}$ which are disjoint to $H_{1},\dots,H_{i-1},H_{i+1},\dots,H_{x}$, $i=1,\dots,x$. It is clear that $\mathcal{S}_{i}$ is a set of $k$-sets in $\Omega_{i}=\Omega\setminus(H_{1}\cup\dots\cup H_{i-1}\cup H_{i+1}\cup\dots\cup H_{x})$ and $|\Omega_{i}|= n-(x-1)k$. Actually $\mathcal{S}_{i}$ is a Cameron-Liebler class of $k$-sets with parameter $1$ in $\Omega_{i}$. Hence, the size of $\mathcal{S}_{i}$ equals $\binom{n-(x-1)k-1}{k-1}$ by the second property in Lemma \ref{basic}. Moreover, the elements of $\mathcal{S}_{i}$ mutually intersect. So, we apply Theorem \ref{ekr} for the set $\mathcal{S}_{i}$ of $k$-subsets of $\Omega_{i}$. Since $n-(x-1)k>2k\Leftrightarrow x<\frac{n}{k}-1$, we know that $\mathcal{S}_{i}\subseteq\mathcal{L}$ is the set of all $k$-sets through a fixed element $p_{i}\in\Omega_i$. Necessarily, $p_{i}\in H_{i}$.
  \par We prove that all $k$-sets in $\mathcal{L}$ pass through at least one of the elements $p_{j}$, $1\leq j\leq x$. Assume that $H'\in\mathcal{L}$ and $p_{j}\notin H'$, for all $j=1,\dots,x$. Denote $|H'\cap H_{i}|$ by $k_{i}$, $i=1,\dots,x$. We know that $k'=\sum^{x}_{i=1}k_{i}\leq k$ since the sets $H_{1},\dots,H_{x}\in\mathcal{L}$ are pairwise disjoint. Since $|\Omega\setminus(H'\cup(\cup^{x}_{i=1}H_{i}))|=n-(x+1)k+k'\geq k+k'$, we can find a $k'$-set $J$ in $\Omega\setminus(H'\cup(\cup^{x}_{i=1}H_{i}))$. Let $\{J_{i}\mid i=1,\dots,x\}$ be a partition of $J$, such that $|J_{i}|=k_{i}$. The sets $H'_{i}=(H_{i}\setminus H')\cup J_{i}$, for $i=1,\dots,x$, are pairwise disjoint $k$-sets. Moreover, $p_{i}\in H_{i}\setminus H'\subseteq H'_{i}$ for all $i=1,\dots,x$. Since $H'_{i}$ and $H_{1}\cup\dots\cup H_{i-1}\cup H_{i+1}\cup\dots\cup H_{x}$ are disjoint, and $p_{i}\in H'_{i}$, the set $H'_{i}$ belongs to $\mathcal{L}$. However, then $H'_{1},\dots,H'_{x},H'$ are $x+1$ pairwise disjoint $k$-sets in $\mathcal{L}$, a contradiction. Hence, all elements of $\mathcal{L}$ are $k$-sets through one of the elements $p_{j}$, $j=1,\dots,x$.
  \par There are $\binom{n-1}{k-1}$ different $k$-sets through $p_{j}$, $j=1,\dots,x$. If $x\geq2$, then there are $k$-sets containing at least two of the elements $p_{j}$, $j=1,\dots,x$, since $k\geq2$, and hence the total number of $k$-sets containing one of the elements $p_{j}$, $j=1,\dots,x$, is smaller than $x\binom{n-1}{k-1}$. However, all $k$-sets in $\mathcal{L}$ contain at least one of the elements $p_{j}$, $j=1,\dots,x$, and $|\mathcal{L}|=x\binom{n-1}{k-1}$. We find a contradiction. So, $x=1$ and in this case, $\mathcal{L}$ consists of all $k$-sets through $p_{1}$.
\end{proof}

\begin{lemma}\label{main2}
  Let $\mathcal{L}$ be a Cameron-Liebler class of $k$-sets with parameter $\frac{n}{k}-1$ in a set $\Omega$ of size $n\geq3k$, with $k\geq2$. Then, $\mathcal{L}$ is the set of all $k$-sets not through a fixed element.
\end{lemma}
\begin{proof}
  By the last property of Lemma \ref{basic}, the set $\overline{\mathcal{L}}$ of $k$-subsets of $\Omega$ not belonging to $\mathcal{L}$ is a Cameron-Liebler class of $k$-sets with parameter $1$. By Lemma \ref{main1} $\overline{\mathcal{L}}$ is the set of all $k$-subsets of $\Omega$ containing a fixed element $p$. Consequently, $\mathcal{L}$ is the set of all $k$-subsets of $\Omega$ not containing $p$.
\end{proof}

\begin{proof}[Proof of Theorem \ref{main}]
  This result combines the results of Lemma \ref{main1} and Lemma \ref{main2}.
\end{proof}

We end this paper with the discussion of a few cases that are not covered by the main theorem.

\begin{remark}
  Let $\Omega$ be a set of size $n$, and let $k$ be a divisor of $n$. Theorem \ref{main} does not cover the cases $k=1$, and $n\in\{k,2k\}$. 
  \begin{itemize}
    \item Assume $k=1$, then any set of $x$ different $1$-subsets of $\Omega$ is a Cameron-Liebler class of $k$-sets with parameter $x$. So, in this case each value $x$, with $0\leq x\leq n$, is admissable as parameter of a Cameron-Liebler class.
    \item If $n=k$, there is only one subset of size $k$, and thus all Cameron-Liebler classes of $k$-sets are trivial.
    \item If $n=2k$, each $k$-uniform partition consists of two $k$-sets which are the complement of each other. Every set of $k$-subsets that is constructed by picking one of both $k$-sets from each $k$-uniform partition, is a Cameron-Liebler class of $k$-sets with parameter $1$, equivalently, it is an Erd\H{o}s-Ko-Rado set. There are $2^{\binom{2k-1}{k-1}}$ different choices to pick $\binom{2k-1}{k-1}$ different $k$-sets, but many choices give rise to isomorphic examples. For $k=1,2,3$, there are $1$, $2$ and $11$ nonisomorphic examples, respectively.
  \end{itemize}
\end{remark}

\paragraph*{Acknowledgment:} The research of Maarten De Boeck is supported by the BOF-UGent (Special Research Fund of Ghent University). The research of Andrea \v{S}vob has been partially supported by the Croatian Science Foundation under the project 1637 and has been supported by a grant from the COST project {\em Random network coding and designs over GF(q)} (COST IC-1104). The authors thank the referees for their suggestions for improvements to the original version of this article.

\end{document}